\documentclass[a4paper,oneside,reqno,12pt]{amsart}
\usepackage[utf8]{inputenc}
\usepackage[T1]{fontenc}
\usepackage{lmodern}
\usepackage{amsmath,amssymb,amsfonts,mathtools}
\usepackage{fullpage}
\usepackage[english]{babel}
\usepackage{xcolor,graphicx,tikz}
\usetikzlibrary{calc,arrows.meta}
\usepackage[font=small,labelfont=bf]{caption}
\usepackage{hyperref}
\hypersetup{pdftitle={KKM meets colorful Caratheodory},pdfauthor={Alexander Polyanskii},colorlinks=true,urlcolor=blue,linkcolor=red!65!black,citecolor=green!40!black}
\newtheorem{theorem}{Theorem}
\newtheorem{lemma}[theorem]{Lemma}

\theoremstyle{remark}

\newtheorem*{remark*}{Remark}
\DeclareMathOperator{\conv}{conv}
\DeclareMathOperator{\aff}{aff}
\DeclareMathOperator{\verts}{vert}
\title[KKM meets Carath\'eodory]{KKM meets colorful Carath\'eodory}
\author{Alexander~Polyanskii}
\address{Alexander Polyanskii, Department of Mathematics, Emory University, 400 Dowman Drive, Atlanta, GA 30322, USA}
\email{apolian@emory.edu}
\keywords{KKM lemma, colorful Carath\'eodory theorem, simplicial complex}
\begin{document}
\begin{abstract}
The colorful Carath\'eodory theorem states that if
$V_1,\dots,V_{d+1}\subset\mathbb R^d$ satisfy
$0\in\conv V_i$ for every $i$, then there exist points
$v_i\in V_i$ such that
$0\in\conv\{v_1,\dots,v_{d+1}\}$.

We give a new proof based on triangulations and the classical KKM lemma.
We lift the sets $V_i$ to parallel affine copies of $\mathbb R^d$ in
$\mathbb R^d\times\mathbb R^d$ and use a triangulation of the resulting
point configuration to construct a KKM cover. Using the same approach, we prove the very colorful Carath\'eodory theorem.
\end{abstract}
\maketitle
\section{Introduction}\label{sec:introduction}

In this paper, we study properties of convex hulls of finite sets in Euclidean spaces, although the results remain true for infinite sets as well. If the dimension of the ambient space is not important to the discussion, we simply refer to it as ``a Euclidean space''; otherwise, we specify its dimension.

For a finite set $X$ in a Euclidean space, let $\conv X$ denote its convex hull.
We write $0$ for the origin. Carath\'eodory's theorem~\cite{Caratheodory1907} bounds the number of points of $X$ needed to represent a point of $\conv X$.

\begin{theorem}[Carath\'eodory's theorem]\label{thm:caratheodory}
Let $V\subset\mathbb R^d$ be a finite set such that $0\in\conv V$.
Then there exist points $v_1,\dots,v_{d+1}\in V$, not necessarily distinct, such that
\[
0\in\conv\{v_1,\dots,v_{d+1}\}.
\]
\end{theorem}

Our aim is to give a new proof of the following generalization, due to B\'ar\'any~\cite{Barany1982}.

\begin{theorem}[Colorful Carath\'eodory theorem]\label{thm:colorful}
Let $V_1,\dots,V_{d+1}\subset\mathbb R^d$ be finite sets such that $0\in\conv V_i$ for every $i$.
Then there exist points $v_1\in V_1,\dots,v_{d+1}\in V_{d+1}$ such that
\[
0\in\conv\{v_1,\dots,v_{d+1}\}.
\]
\end{theorem}

Our approach is to lift the color classes $V_i$ from $\mathbb R^d$ to $\mathbb R^d\times \mathbb R^d\simeq \mathbb R^{2d}$ by considering the Cartesian products $V_i\times \{a_i\}$, where $a_1,\dots,a_{d+1}\in \mathbb R^d$ are any $d+1$ affinely independent points. We then triangulate the convex hull of the union of the resulting sets. The second ingredient is the facet form of the Knaster--Kuratowski--Mazurkiewicz (KKM) lemma~\cite{KnasterKuratowskiMazurkiewicz1929}.

\begin{theorem}[KKM lemma]\label{thm:kkm}
Let $\sigma$ be a $k$-dimensional simplex in a Euclidean space with facets $\sigma_1,\dots,\sigma_{k+1}$.
Let $X_1,\dots,X_{k+1}$ be closed sets in the same space satisfying
\[
\sigma\subseteq X_1\cup\cdots\cup X_{k+1}
\qquad\text{and}\qquad \sigma_i\subseteq X_i\quad\text{for every }i.
\]
Then $\sigma\cap X_1\cap\cdots\cap X_{k+1}\ne\varnothing$.
\end{theorem}

The connection between combinatorial topology and the colorful Carath\'eodory theorem is well known. In particular, the theorem follows from Meshulam's lemma, a homological generalization of Sperner's lemma; see the survey by De Loera, Goaoc, Meunier, and Mustafa~\cite[Proposition~2.6 and Section~3.1]{DeLoeraGoaocMeunierMustafa2019}. Another approach using KKM-type results was given by Frick and Zerbib~\cite{FrickZerbib2019}. Their proof relies on a generalization of the KKM lemma, while our proof uses its classical form.

Section~\ref{sec:notation} collects the definitions and facts about triangulations used in the proofs. In Section~\ref{sec:colorful}, we prove the colorful Carath\'eodory theorem (Theorem~\ref{thm:colorful}). In Section~\ref{sec:very-colorful}, we adapt the argument to prove the very colorful Carath\'eodory theorem (Theorem~\ref{thm:very-colorful}).

\section{Notation}\label{sec:notation}

An \emph{affine combination} of points $x_1,\dots,x_m$ in a Euclidean space is a sum $\sum_{j=1}^m\lambda_jx_j$ with real coefficients satisfying $\sum_{j=1}^m\lambda_j=1$.
For a finite set $X$ in a Euclidean space, its \emph{affine hull}, denoted by $\aff X$, is the set of all affine combinations of points from $X$.
We put $\aff\varnothing=\varnothing$.

A \emph{convex combination} of points $x_1,\dots,x_m$ in a Euclidean space is a sum $\sum_{j=1}^m\lambda_jx_j$ with $\lambda_j\geq 0$ and $\sum_{j=1}^m\lambda_j=1$.
For a finite set $X$ in a Euclidean space, its \emph{convex hull}, denoted by $\conv X$, is the set of all convex combinations of points of $X$. We put $\conv\varnothing=\varnothing$. The convex hull of a finite set is closed.
A convex combination is \emph{positive} if all its coefficients are strictly positive.

A finite set $\{v_1,\dots,v_{k+1}\}$ in a Euclidean space is called \emph{affinely independent} if the vectors $v_2-v_1,\dots,v_{k+1}-v_1$ are linearly independent.
A \textit{simplex} is the convex hull of an affinely independent set in a Euclidean space; this set is called its \emph{vertex set}.
For a simplex~$\sigma$, we denote its vertex set by $\verts \sigma$.
By definition, if $\sigma\subseteq \mathbb R^n$, then it has at most $n+1$ vertices.
Since the vertices of $\sigma$ are affinely independent, every point of $\sigma$ has a unique representation as a convex combination of its vertices.
We regard the empty set as a simplex with an empty vertex set.

A \emph{face} of $\sigma$ is the convex hull of a subset of its vertices.
It is called proper if it is distinct from $\sigma$.
The \emph{relative interior} of a nonempty simplex $\sigma$, denoted by $\sigma^\circ$, consists of all \emph{positive} convex combinations of its vertices.
In particular, for a point $v$, we have $\{v\}^\circ=\{v\}$; we put $\varnothing^\circ=\varnothing$.
Since every point of $\sigma$ has a unique representation as a convex combination of its vertices, $\sigma$ is the disjoint union of the relative interiors of all its faces:
\begin{equation}
\label{equation sigma is the disjoint union of relative interiors of its faces}
\sigma=\bigsqcup_{\tau\text{ a face of }\sigma}\tau^\circ.
\end{equation}

A \emph{simplicial complex} is a finite family $\mathcal T$ of simplices such that every face of a simplex in $\mathcal T$ belongs to $\mathcal T$, and the intersection of any two simplices in $\mathcal T$ is a face of each.
We write $|\mathcal T|:=\bigcup_{\sigma\in\mathcal T}\sigma$ for its \emph{underlying set} and $\verts\mathcal T=\{v:\{v\}\in\mathcal T\}$ for its \emph{vertex set}.
Since the underlying set is a finite union of closed simplices, it is also closed.
A \emph{subcomplex} is a subfamily that is itself a simplicial complex.

Moreover,
\begin{equation}\label{eq:partition}
|\mathcal T|=\bigcup_{\sigma\in\mathcal T}\sigma\overset{\eqref{equation sigma is the disjoint union of relative interiors of its faces}}{=}\bigcup_{\sigma\in\mathcal T}\sigma^\circ=\bigsqcup_{\sigma\in\mathcal T}\sigma^\circ,
\end{equation}
where the last equality follows from $\sigma_1^\circ\cap\sigma_2^\circ=\varnothing$ for distinct $\sigma_1,\sigma_2\in\mathcal T$.
Indeed, suppose that $\sigma_1^\circ\cap\sigma_2^\circ\ne \varnothing$ and let $x$ be any point in this intersection.
Then $x$ lies in $\sigma_1\cap\sigma_2$, which is, by definition, a common face of $\sigma_1$ and $\sigma_2$.
By~\eqref{equation sigma is the disjoint union of relative interiors of its faces} applied to the simplex $\sigma_1\cap\sigma_2$, the point $x$ belongs to $\tau^\circ$ for some common face $\tau$ of $\sigma_1$ and $\sigma_2$.
Since $\sigma_1\ne\sigma_2$, the face $\tau$ is proper in at least one of them, say, $\sigma_1$.
By~\eqref{equation sigma is the disjoint union of relative interiors of its faces} applied to $\sigma_1$, the sets $\sigma_1^\circ$ and $\tau^\circ$ are disjoint, contradicting $x\in\sigma_1^\circ\cap\tau^\circ$.

For a finite set $V$ in a Euclidean space, a simplicial complex is called a \textit{triangulation of $V$} if its vertex set is $V$ and its underlying set is $\conv V$. We finish this section with a few standard lemmas about triangulations, whose proofs are omitted.

\begin{lemma}\label{lemma existence of triangulation}
Every finite set in a Euclidean space has a triangulation.
\end{lemma}

Lemma~\ref{lemma existence of triangulation} immediately implies Theorem~\ref{thm:caratheodory}.
Indeed, suppose that $0\in \conv V$ for a finite set $V\subset \mathbb R^d$.
By Lemma~\ref{lemma existence of triangulation}, there is a triangulation $\mathcal T$ of $V$.
By~\eqref{eq:partition}, the origin lies in the relative interior of a unique simplex in $\mathcal T$ whose vertices belong to $V$.
Recall that a simplex in $\mathbb R^d$ has at most $d+1$ vertices; repeating a vertex if necessary gives the stated conclusion.

\begin{lemma}\label{lemma supporting hyperplane and triangulation}
    Let $\mathcal T$ be a triangulation of a finite set $V$ in a Euclidean space. Suppose that $V$ lies in a closed halfspace bounded by a hyperplane $H$. Then $\{\sigma \in \mathcal T: \verts \sigma \subseteq V\cap H\}$ is a triangulation of $V\cap H$.
\end{lemma}

To prove the very colorful Carath\'eodory theorem, we use the following extension of Lemma~\ref{lemma existence of triangulation}.

\begin{lemma}
\label{lemma extension of triangulation}
    For a finite set $V$ in a Euclidean space, let $\mathcal T$ be a triangulation of $V$.
    Let $V'$ be any finite set of points in the same space that lie outside $\conv V$.
    Then there is a triangulation of $V\cup V'$ that contains $\mathcal T$ as a subcomplex.
\end{lemma}

We will also use the following simple observation.

\begin{lemma}
\label{lemma intersection of simplices}
    Let $H$ be a hyperplane in $\mathbb R^n$ and let $\sigma\subseteq H$ be a simplex with $n$ vertices. If points $x_1$ and $x_2$ lie in the same open halfspace determined by $H$, then the relative interiors of the simplices $\conv (\{x_1\}\cup \verts \sigma)$ and $\conv (\{x_2\}\cup \verts \sigma)$ intersect.
\end{lemma}

\section{Proof of Theorem~\ref{thm:colorful}}\label{sec:colorful}

We first define the simplex to which we will apply the KKM lemma. Let $A=\{a_1,\dots,a_{d+1}\}$ be any affinely independent set in $\mathbb R^d$. We embed $A$ in $\mathbb R^d\times \mathbb R^d$ as follows:
\[
    \{0\}\times A=\big\{(0,a_1),\dots, (0,a_{d+1})\big\}\subseteq \mathbb R^d\times \mathbb R^d.
\]

For each $i$, we lift $V_i\subset \mathbb R^d$ to $\mathbb R^d\times \mathbb R^d$ by setting
\begin{equation*}
W_i:=V_i\times\{a_i\}\subset\mathbb R^d\times\{a_i\}\subset \mathbb R^d\times \mathbb R^d.
\end{equation*}
Put
\begin{equation*}
W:=W_1\sqcup\dots\sqcup W_{d+1}
\subset\mathbb R^d\times A
\subset\mathbb R^d\times\mathbb R^d.
\end{equation*}

By Lemma~\ref{lemma existence of triangulation}, there is a triangulation $\mathcal T$ of $W$.
By~\eqref{eq:partition}, we have
\begin{equation}
\label{eq:conv-w}
\conv W
=
\bigsqcup\big\{\sigma^\circ:\sigma\in\mathcal T\big\}.
\end{equation}

We now define the closed sets to which we will apply the KKM lemma. For each $i$, let
\begin{equation*}
U_i
:=
\bigsqcup\big\{
\sigma^\circ:\sigma\in\mathcal T
\text{ with at most one vertex in }W_i
\big\}.
\end{equation*}
The condition on the vertices is inherited by faces, so these simplices form a subcomplex of $\mathcal T$. Hence
\begin{equation*}
U_i
=
\bigcup\big\{
\sigma:\sigma\in\mathcal T
\text{ with at most one vertex in }W_i
\big\},
\end{equation*}
and $U_i$ is closed.

Each simplex in $\mathcal T$ has at most $2d+1$ vertices, so it has at most one vertex in at least one of the sets $W_1,\dots,W_{d+1}$.
Together with~\eqref{eq:conv-w} and the definition of $U_i$, this gives
\begin{equation}
\label{eq:cover-w}
\conv W=U_1\cup\dots\cup U_{d+1}.
\end{equation}

For each $i$, put
\[
H_i:=\mathbb R^d\times \aff (A\setminus \{a_i\}).
\]
By definition, $H_i\cap W=W\setminus W_i$. Moreover, $W_i=V_i\times \{a_i\}\subset \mathbb R^d\times\{a_i\}$ lies entirely in one of the open halfspaces determined by $H_i$. Lemma~\ref{lemma supporting hyperplane and triangulation} therefore implies that $\{\sigma \in \mathcal T:\verts \sigma\subseteq W\setminus W_i\}$ is a triangulation of $W\setminus W_i$. Thus
\begin{equation*}
\conv(W\setminus W_i)
=
\bigsqcup\big\{
\sigma^\circ:\sigma\in\mathcal T
\text{ with no vertex in }W_i
\big\}
\subseteq U_i,
\end{equation*}
where the inclusion follows from the definition of $U_i$.

We are now ready to verify the conditions of the KKM lemma. Since $0\in\conv V_i$ for each $i$, we have
\begin{equation*}
\conv(\{0\}\times A)
\subseteq
\conv W
=
U_1\cup\dots\cup U_{d+1}
\end{equation*}
and
\begin{equation*}
\conv\bigl(\{0\}\times(A\setminus\{a_i\})\bigr)
\subseteq
\conv(W\setminus W_i)
\subseteq U_i
\end{equation*}
for each $i$. Thus the simplex $\conv(\{0\}\times A)$ with $d+1$ vertices and the sets
$U_i$ satisfy the conditions of Theorem~\ref{thm:kkm}.
Applying that theorem, we obtain
\begin{equation*}
\conv(\{0\}\times A)
\cap U_1\cap\dots\cap U_{d+1}
\ne\varnothing.
\end{equation*}
Choose $a\in\conv A$ such that $(0,a)$ lies in this intersection.

Since
\begin{equation*}
U_1\cap\dots\cap U_{d+1}
=
\bigsqcup\big\{
\sigma^\circ:\sigma\in\mathcal T
\text{ with at most one vertex in }W_i
\text{ for each }i
\big\},
\end{equation*}
the unique simplex $\tau \in \mathcal T$ whose relative interior contains $(0,a)$ has at most one vertex in each $W_i$.
Adding a point from each missing $W_j$ to $\verts \tau$, we obtain a set of size $d+1$ points, one from each $W_i$. Denoting the corresponding point from $W_i$ in this set by $(v_i,a_i)$, where $v_i\in V_i$, we obtain
\begin{equation*}
(0,a)
\in\tau^\circ\subseteq
\conv\{(v_1,a_1),\dots,(v_{d+1},a_{d+1})\}.
\end{equation*}
Projecting this inclusion onto the first factor completes the proof.

\section{Proof of the very colorful Carath\'eodory theorem}\label{sec:very-colorful}
We now prove the following strengthening \cite{HolmsenPachTverberg2008,ArochaEtAl2009}.

\begin{theorem}[Very colorful Carath\'eodory theorem]\label{thm:very-colorful}
Let $V_1,\dots,V_{d+1}\subset\mathbb R^d$ be nonempty finite sets such that $0\in\conv(V_i\cup V_j)$ for every distinct $i$ and $j$.
Then there exist points $v_1\in V_1,\dots,v_{d+1}\in V_{d+1}$ such that
\begin{equation*}
0\in\conv\{v_1,\dots,v_{d+1}\}.
\end{equation*}
\end{theorem}

\begin{proof}
Define $A$, $W_i$, $W$, $\mathcal T$, $U_i$, and $H_i$ as in the proof of Theorem~\ref{thm:colorful}. We now modify some of the sets $W_i$ by adding new points.

Call an index $i$ \emph{avoiding} if $0\notin\conv V_i$, and associate with it the \emph{avoiding point} $\tilde w_i:=(0,a_i)$. Let
\begin{equation*}
\tilde{W}_i:=
\begin{cases}
    W_i\cup\{\tilde{w}_i\}, & \text{if $i$ is avoiding};\\
    W_i, & \text{otherwise}.
\end{cases}
\end{equation*}
Put $\tilde{W}=\tilde{W}_1\sqcup\dots\sqcup\tilde{W}_{d+1}$.

For each avoiding $i$, we have $\tilde w_i\notin\conv W$. Indeed, affine independence of $A$ implies that any convex representation of a point with second coordinate $a_i$ by points of $W$ can use only points of $W_i$. Thus $\tilde w_i\in\conv W$ would imply $0\in\conv V_i$, contrary to the assumption that $i$ is avoiding.

By Lemma~\ref{lemma extension of triangulation}, there is a triangulation $\tilde{\mathcal T}$ of $\tilde W$ containing $\mathcal T$ as a subcomplex, with
\begin{equation}
\label{eq:conv-tilde-w}
\conv\tilde{W}=\bigsqcup\big\{\sigma^\circ:\sigma\in\tilde{\mathcal T}\big\}.
\end{equation}

For distinct avoiding indices $i$ and $j$, any convex representation of $0$ by points of $V_i\cup V_j$ uses both color classes with positive total weights. After lifting, this gives a point of $\conv(W_i\cup W_j)$ in the relative interior of the line segment connecting $\tilde w_i$ and $\tilde w_j$. Thus this line segment cannot be a simplex of $\tilde{\mathcal T}$: its relative interior meets $\conv W=|\mathcal T|$, and hence it meets the relative interior of a simplex of $\mathcal T$, contradicting the disjointness of relative interiors in $\tilde{\mathcal T}$. Consequently, no simplex of $\tilde{\mathcal T}$ contains two avoiding points. Moreover, any simplex of $\tilde{\mathcal T}$ whose vertices all belong to $W$ is already a simplex of $\mathcal T$, by the same disjointness argument. Therefore, every simplex in $\tilde{\mathcal T}\setminus\mathcal T$ has exactly one avoiding vertex and is of the form
$\conv(\{\tilde w_i\}\cup\verts \sigma)$ with $\sigma\in\mathcal T$. By~\eqref{eq:conv-w} and~\eqref{eq:conv-tilde-w}, its relative interior is disjoint from $\conv W=|\mathcal T|$.

For each $i$, define
\begin{equation*}
\tilde{U}_i
:=
\bigsqcup\big\{
\sigma^\circ:\sigma\in\tilde{\mathcal T}
\text{ with at most one vertex in }W_i
\text{ and no vertices in }\tilde{W}_i\setminus W_i
\big\}.
\end{equation*}
The defining conditions are inherited by faces, so these simplices form a subcomplex of $\tilde{\mathcal T}$. Hence $\tilde U_i$ is closed.

We claim that
\begin{equation}\label{eq:new-cover}
\conv\tilde{W}=\tilde{U}_1\cup\dots\cup\tilde{U}_{d+1}.
\end{equation}
By~\eqref{eq:cover-w}, the sets $U_i\subseteq\tilde U_i$ already cover $\conv W$.
In view of~\eqref{eq:conv-tilde-w} and the definition of $\tilde{U}_i$, it remains to show that for each new simplex
$\tau=\conv(\{\tilde w_i\}\cup\verts \sigma)\in \tilde{\mathcal T}\setminus \mathcal T$, where $\sigma\in\mathcal T$, its relative interior $\tau^\circ$ is contained in $\tilde U_j$ for some $j$.
Suppose that $\tau^\circ$ is not contained in any $\tilde U_j$ with $j\ne i$.
Since $\tilde w_i$ is the only avoiding vertex of $\tau$, this means that $\tau$ has at least two vertices in each $W_j$ with $j\ne i$.
It has at most $2d+1$ vertices in total, so $\sigma$ has exactly two vertices in the set $W_j$ for each $j\ne i$ and none in $W_i$. Therefore, $\sigma$ lies in the hyperplane $H_i=\mathbb R^d\times \aff(A\setminus \{a_i\})$. Let $w_i$ be any point of $W_i$. Since $w_i$ and $\tilde w_i$ lie in the same open halfspace determined by $H_i$, Lemma~\ref{lemma intersection of simplices} implies that
\[
\tau^\circ\cap\conv(\{w_i\}\cup\verts\sigma)^\circ\ne\varnothing.
\]
Since $\{w_i\}\cup\verts\sigma\subset W$, this means that $\tau^\circ$ intersects $\conv W$, a contradiction. This proves~\eqref{eq:new-cover}.

For every $i$, Lemma~\ref{lemma supporting hyperplane and triangulation} gives
\[
\conv(\tilde W\setminus\tilde W_i)\subseteq\tilde U_i,
\qquad (0,a_i)\in\conv\tilde W_i,
\]
where the second assertion follows from the definitions of an avoiding index and of $\tilde W_i$. Together with~\eqref{eq:new-cover}, these are precisely the covering and facet conditions used in the proof of Theorem~\ref{thm:colorful}. Repeating its final KKM argument, with the corresponding objects replaced by their tilded versions, yields a simplex with at most one vertex in each $W_i$ and, by the definition of $\tilde U_i$, none of the added avoiding points. Adding one vertex from each missing $W_i$ and projecting onto the first factor completes the proof.
\end{proof}

\subsection*{Acknowledgements}
The author is grateful to Mikhail Bludov and Aliaksei Vasileuski for discussions that inspired the proofs presented here.

No AI model is used in the discovery or proof of any mathematical result presented in this paper.

\subsection*{Funding}
A.P. is partially supported by the NSF grant DMS 2349045.

\end{document}